\documentclass[12pt,a4paper]{article}
\usepackage{amsmath, enumerate, fullpage,wasysym, bbm,bm}
\usepackage{amsthm}
\usepackage{amssymb}
\usepackage{amsbsy}
\usepackage{amsfonts}
\usepackage{amstext}
\usepackage{amscd}
\usepackage[dvips]{graphicx,psfrag}
\numberwithin{equation}{section}
\theoremstyle{plain}
\newtheorem{theorem}{Theorem}[section]
\newtheorem{proposition}[theorem]{Proposition}

\newtheorem{lemma}[theorem]{Lemma}
\theoremstyle{definition}
\newtheorem{example}[theorem]{Example}

\newtheorem{remark}[theorem]{Remark}
\newtheorem{definition}[theorem]{Definition}

\newcommand{\real}{\mathbb{R}}
\newcommand{\comp}{\mathbb{C}}

\begin{document}
\title{On free infinite divisibility for classical Meixner distributions}

\author{Marek Bo\.zejko\footnote{email: marek.bozejko@math.uni.wroc.pl } \\ Institute of Mathematics,
University of Wroc{\l}aw, \\ Pl.\ Grunwaldzki 2/4, 50-384 Wroc{\l}aw, Poland \\ 
and \\ Takahiro Hasebe\footnote{Current address is: Laboratoire de Math\'ematiques, Universit\'e de Franche-Comt\'e, 25030 Besan\c{c}on cedex, France. Email: thasebe@univ-fcomte.fr} \\ 
Graduate School of Science,  Kyoto University, \\  Kyoto 606-8502, Japan}   
\date{}
\maketitle
\begin{abstract}
We prove that symmetric Meixner distributions, whose probability densities are proportional to $|\Gamma(t+ix)|^2$, are freely infinitely divisible  for $0<t\leq\frac{1}{2}$. 
The case $t=\frac{1}{2}$ corresponds to the law of L\'evy's stochastic area whose probability density is $\frac{1}{\cosh(\pi x)}$. 
A logistic distribution, whose probability density is proportional to $\frac{1}{\cosh^2(\pi x)}$, is freely infinitely divisible too.   
\end{abstract}



\section{Introduction}\label{sec1}
The free convolution $\mu \boxplus \nu$ of probability measures $\mu$ and $\nu$ on $\mathbb{R}$ is the distribution of $X+Y$, where 
$X$ and $Y$ are free self-adjoint random variables  respectively following the distributions $\mu$ and $\nu$.  A probability measure $\nu$ on $\mathbb{R}$ is said to be \emph{freely infinitely divisible} if, for any $n \in \{1,2,3,\cdots \}$, there exists $\nu_n$ such that 
$$
\nu = \underbrace{\nu_n \boxplus \cdots \boxplus \nu_n.}_{\text{$n$ times}}
$$
This concept was introduced in \cite{V86} and its basic characterization was established in \cite{BV93}. The most important freely infinitely divisible distributions are Wigner's semicircle law and the free Poisson law. 

Recent work has increased examples of probability measures which are infinitely divisible in both senses, classical and free: the Gaussian distribution \cite{BBLS11}, 
chi-square distribution $\frac{1}{\sqrt{\pi x}}e^{-x}1_{[0,\infty)}(x)\,dx$ \cite{AHS}, positive Boolean stable law with stability index $\alpha \in (0,\frac{1}{2}]$ \cite{AHb} and Student distribution $\frac{1}{B(\frac{1}{2}, n-\frac{1}{2})} \frac{1}{(1+x^2)^n}\,1_\real(x)\,dx$ for $n=1,2,3,\cdots$ \cite{H}. It is not yet clear whether
a general theory of the intersection of free and classical infinite divisibility exists. We will add two more examples, Meixner distributions and the logistic distribution, which may contribute to a solution.  

We will prove that symmetric  \emph{Meixner distributions} 
$$
\rho_t(dx):= \frac{4^t}{2\pi \Gamma(2t)}|\Gamma(t+ix)|^2\,dx,~~x \in \real
$$
are freely infinitely divisible for $0 < t \leq \frac{1}{2}$, where $\Gamma(z)$ is the gamma function defined by: 
$$
\Gamma(z)= \int_{0}^\infty t^{z-1}e^{-t}\,dt,~~~z>0. 
$$
The gamma function satisfies the functional relation $\Gamma(z+1)=z\Gamma(z)$, which extends $\Gamma$ to a meromorphic function in $\comp$ with poles at $z=0,-1,-2,-3,\ldots$ \cite[Chapter 6]{AS70}.  
 The measures $\rho_t$ are probability distributions of a L\'evy process, called a Meixner process \cite{ST98}, since the characteristic function of $\rho_t$ is given by 
\begin{equation}\label{eq90}
\widehat{\rho}_t(z)= \left(\frac{1}{\cosh(\frac{z}{2})}\right)^{2t}.    
\end{equation}
Hence $\rho_t$ is classically infinitely divisible for any $t>0$. The measure $\rho_t$ orthogonalizes Meixner-Pollaczek polynomials $\{P^{(t)}_n(x)\}_{n=0}^\infty$ which satisfy the recurrence relation \cite{KLS10} 
$$
xP_n^{(t)}(x) = P_{n+1}^{(t)}(x) + \frac{n(n+2t-1)}{4}P_{n-1}^{(t)}(x),~~n\geq1,
$$
with initial conditions $P_0^{(t)}(x)=1,~P_1^{(t)}(x)=x.$

If $t=\frac{1}{2}$, the measure $\rho_{1/2}$ coincides with 
$$
\mu_1(dx)=\frac{1}{\cosh (\pi x)}\,dx,~~x\in \real, 
$$ 
which is  the law of \emph{L\'evy's stochastic area}\footnote{This measure is also called the \emph{hyperbolic secant distribution}. }  
$$
\frac{1}{2}\int_0^1(B^1_t \,dB^2_t - B^2_t \,dB^1_t), 
$$
where $(B^1_t, B^2_t)$ is a standard two-dimensional Brownian motion~\cite{L54}. The moments $m_{n}$ of the rescaled measure 
$\displaystyle \frac{1}{2\cosh(\pi x /2)}\,dx$ are \emph{Euler numbers} (with positive signs): 
\[
(m_0,m_2,m_4,m_6,m_8,\cdots)=(1,1,5,61,1385,50521,\cdots),~~m_{2n+1}=0, ~n \geq 0.   
\] 
See \cite[Chapter 23]{AS70} for Euler numbers. 

The \emph{logistic distribution}  
$$
\mu_2(dx)=\frac{\pi}{2\cosh^2(\pi x)}\,dx,~~x\in \real, 
$$ 
is know to be classically infinitely divisible \cite{B92}, and we are going to prove that it is freely infinitely divisible too. 
This measure orthogonalizes  \emph{continuous Hahn polynomials} $\{P_n(x)\}_{n =0}^\infty$ which satisfy the recurrence relation \cite{KLS10}
$$
xP_n(x) = P_{n+1}(x) + \frac{n^4}{4(4n^2-1)}P_{n-1}(x),~~n\geq1,
$$
with initial conditions $P_0(x)=1,~P_1(x)=x.$

The moments $m'_{n}$ of the rescaled measure $\frac{\pi}{4\cosh^2 (\pi x /2)}\,dx$ are  
$$
(m'_0,m'_2,m'_4,m'_6,m'_8,\cdots)=\left(1,\frac{1}{3}, \frac{7}{15},\frac{31}{21}, \frac{127}{15},\cdots\right),~~m'_{2n+1}=0, ~n \geq 0, 
$$
which can be written as $m_{n}'= |(2-2^{n})B_{n}|$ in terms of \emph{Bernoulli numbers} $B_n$ \cite{AS70}.

\section{Preliminaries}
Let $\comp^+$ and $\comp^-$ be the upper half-plane and the lower half-plane respectively. 
Basic tools for proving free infinite divisibility of a probability measure $\mu$ are the Cauchy transform 
\[
G_\mu(z):=\int_{\mathbb{R}}\frac{1}{z-x}\, \mu(dx),~~z\in\mathbb{C}^+
\]
and its reciprocal $F_\mu(z):=\frac{1}{G_\mu(z)}$. Let $\Gamma_{\alpha,M}$ be a truncated cone 
$$
\Gamma_{\alpha,M}:= \{z \in \mathbb{C}^+: \text{\normalfont Im}\,z >M,~ |\text{\normalfont Re}\,z| <\alpha \text{\normalfont Im}\,z\},~~\alpha,M>0. 
$$
The reciprocal Cauchy transform maps $\comp^+$ to $\comp^+$ analytically, and it satisfies $\text{Im}\,F_\mu(z) \geq  \text{Im}\,z$ for $z \in \comp^+.$
For any $0 < \varepsilon < \alpha$ and $\mu$, there exist $M>0$ and  a unique univalent inverse map $F_\mu^{-1}$ from $\Gamma_{\alpha-\varepsilon, (1+\varepsilon)M}$ into 
$\comp^+$ such that $F_\mu(\Gamma_{\alpha, M}) \supset \Gamma_{\alpha-\varepsilon, (1+\varepsilon)M}$ and $F_\mu \circ F_\mu^{-1}=\text{\normalfont Id}$ in 
$\Gamma_{\alpha-\varepsilon, (1+\varepsilon)M}$ \cite{BV93}. 

Free convolution and free infinite divisibility can be characterized by the \emph{Voiculescu transform} of $\mu$ defined by 
\begin{equation}\label{eq44}
\phi_\mu(z) := F_\mu^{-1}(z)-z
\end{equation}
in a domain of the form $\Gamma_{\beta,L}$.  

\begin{theorem}[\cite{BV93}]
\begin{enumerate}[\rm(1)]
\item The free convolution $\mu \boxplus \nu$ is a unique probability measure such that 
\[
\phi_{\mu \boxplus \nu}(z) = \phi_\mu (z) + \phi_\nu(z)
\]
in a common domain of the form $\Gamma_{\beta,L}$. 

\item A probability measure $\mu$ on $\real$ is freely infinitely divisible if and only if $-\phi_\mu$ analytically extends to a Pick function, i.e.\ an analytic function which maps $\comp^+$ into $\comp^+ \cup \real$. 
\end{enumerate}
\end{theorem}

In terms of analytic properties of $F_\mu^{-1}$, a useful subclass of freely infinitely divisible distributions is introduced.  
\begin{definition}
A probability measure $\mu$ is said to be in the class $\mathcal{UI}$ if $F_\mu^{-1}$ defined in a domain of the form $\Gamma_{\beta,L}$ analytically extends to a univalent map in $\comp^+.$ 
Equivalently, $\mu \in \mathcal{UI}$ if and only if  there exists a simply connected open set $\comp^+ \subset \Omega \subset \comp$ such that 
\begin{enumerate}[\rm(i)]
\item $F_\mu$ analytically extends to a univalent map in $\Omega$, 
\item $F_\mu(\Omega) \supset \comp^+$. 
\end{enumerate}
This equivalence is proved just by applying Riemann mapping theorem. 
\end{definition}
\begin{remark}\label{rem56}
In \cite{AHa} we required $F_\mu$ to be univalent in $\comp^+$ in the definition of $\mu \in \mathcal{UI}$, but this automatically follows. If $F^{-1}_\mu$ is analytic in $\comp^+$, then $F^{-1}_\mu \circ F_\mu (z) =z$ for $z \in \comp^+$ by Identity Theorem, so that $F_\mu$ is univalent in $\comp^+$.    
\end{remark}

\begin{lemma}[\cite{AHa}]\label{lem1}
\begin{enumerate}[\rm(1)]
\item If $\mu \in \mathcal{UI}$, then $\mu$ is freely infinitely divisible. 
\item The class $\mathcal{UI}$ is closed with respect to the weak convergence. 
\item The class $\mathcal{UI}$ is not closed under free convolution, i.e.\ $\mu,\nu \in \mathcal{UI}$ does not imply $\mu \boxplus \nu \in \mathcal{UI}.$ 
\end{enumerate}
\end{lemma}
This class was essentially introduced in \cite{BBLS11} to show that the normal law is freely infinitely divisible, and this class has been successfully applied to several probability measures \cite{ABBL10, AB, AHa, AHb, H}. Examples are presented below, mostly taken from the aforementioned references.
\begin{example}The following probability measures belong to $\mathcal{UI}$. 
\begin{enumerate}[\rm(1)]
\item Wigner's semicircle law 
$$
\displaystyle \mathbf{w}(dx)= \frac{1}{2\pi}\sqrt{4-x^2}\,1_{[-2,2]}(x)\,dx,~~ F^{-1}_{\mathbf{w}}(z)=z+\frac{1}{z}.
$$ 
\item The free Poisson law (or Marchenko-Pastur law) 
$$
\displaystyle \mathbf{m}(dx)= \frac{1}{2\pi}\sqrt{\frac{4-x}{x}}\,1_{(0,4]}(x)\,dx, ~~F^{-1}_{\mathbf{m}}(z)=z+\frac{z}{z-1}.
$$ 
\item The Cauchy distribution 
$$
\displaystyle \mathbf{c}(dx)= \frac{1}{\pi(1+x^2)}\,1_{\real}(x)\,dx,  ~~F^{-1}_{\mathbf{c}}(z)=z-i.
$$
\item \cite{AHa} The beta distribution 
$$
\displaystyle \bm{\beta}_a(dx)= \frac{\sin (\pi a)}{\pi a} \left(\frac{1-x}{x}\right)^a\,1_{(0,1)}(x)\,dx,~~F_{\bm{\beta}_a}^{-1}(z) = \frac{1}{1-(1-\frac{a}{z})^\frac{1}{a}}
$$
for $\frac{1}{2} \leq |a| <1$. $\bm{\beta}_{\frac{1}{2}}$ is equal to $\mathbf{m}$ up to scaling. 
 
\item \cite{BBLS11} The Gaussian distribution
$$
\displaystyle \mathbf{g}(dx)=\frac{1}{\sqrt{2\pi}}e^{-\frac{x^2}{2}}1_\real(x)\,dx.
$$  

\item \cite{ABBL10} The $q$-Gaussian distribution 
$$
\displaystyle \mathbf{g}_q(dx)= \frac{\sqrt{1-q}}{\pi} \sin \theta(x) \prod_{n=1}^\infty (1-q^n)|1-q^n e^{2i\theta(x)}|^2 \,1_{\left[-\frac{2}{\sqrt{1-q}},\frac{2}{\sqrt{1-q}}\right]}(x)\,dx 
$$
for $q \in [0,1)$, where $\theta(x)$ is the solution of $x = \frac{2}{\sqrt{1-q}}\cos \theta$, $\theta \in [0,\pi].$ When $q \to 1$, $\mathbf{g}_q$ converges weakly to $\mathbf{g}$, and  
$\mathbf{g}_0$ coincides with $\mathbf{w}.$
For $q\in (0,1)$, the density function of $\mathbf{g}_q$ can be written as \cite{LM95} 
$$
 \frac{1}{2\pi} q^{-\frac{1}{8}} (1-q)^{\frac{1}{2}} \Theta _{1} \left(\frac{\theta(x)}{\pi}, \frac{1}{2\pi i} \log q\right), 
$$
where $\displaystyle \Theta_1(z,\tau):= 2\sum_{n=0}^\infty (-1)^n (e^{i\pi \tau})^{(n+\frac{1}{2})^2}\sin(2n+1)\pi z$ is a Jacobi theta function.  

\item \cite{AB} The ultraspherical distribution 
$$
\mathbf{u}_n(dx)= \frac{1}{16^n B(n+\frac{1}{2},n+\frac{1}{2})} (4-x^2)^{n-\frac{1}{2}}1_{[-2,2]}(x)\,dx,~~~n=1,2,3,4,\cdots,
$$
where $B(p,q)$ is the beta function. The semicircle law $\mathbf{w}$ appears in the case $n=1$ and the normal law $\mathbf{g}$ in the limit $n \to \infty$ if $\mathbf{u}_n$ are suitably scaled. 

\item \cite{H} The Student distribution 
$$
\displaystyle \mathbf{t}_n(dx) = \frac{1}{B(\frac{1}{2}, n-\frac{1}{2})} \frac{1}{(1+x^2)^n}\,1_\real(x)\,dx, ~~n=1,2,3,\cdots. 
$$ 
$\mathbf{t}_1$ coincides with $\mathbf{c}$, and if suitably scaled, $\mathbf{t}_n$ weakly converge to $\mathbf{g}$ as $n \to \infty$. 
\item \cite{AHb} The Boolean stable law 
$$
\displaystyle \frac{d\mathbf{b}_\alpha^\rho}{dx} =
\begin{cases}
 \dfrac{\sin(\pi\rho \alpha)}{\pi} \dfrac{x^{\alpha-1}}{x^{2\alpha}+2x^\alpha \cos(\pi \rho \alpha)+1},&x>0,\\[15pt]
 \dfrac{\sin(\pi(1-\rho)\alpha)}{\pi} \dfrac{|x|^{\alpha-1}}{|x|^{2\alpha}+2|x|^\alpha \cos(\pi (1-\rho) \alpha)+1},&x<0,\\ 
\end{cases}
$$
for $0<\alpha \leq \frac{1}{2}$, $\rho \in [0,1].$
\end{enumerate}
If $\frac{1}{2} \leq \alpha \leq \frac{2}{3}$ and $2-\frac{1}{\alpha}\leq\rho \leq \frac{1}{\alpha}-1$,  the Boolean stable law $\mathbf{b}_\alpha^\rho$ (defined as above too) is still freely infinitely divisible, but not in the class $\mathcal{UI}$ \cite{AHb}. However, most of the known freely infinitely divisible distributions belong to $\mathcal{UI}$ as presented above.
\end{example}
In order to prove $\mu \in \mathcal{UI}$, the following sufficient condition is useful. 
\begin{proposition} A probability measure $\mu$ on $\real$ is in $\mathcal{UI}$ if there exists a simple, continuous curve $\gamma =(\gamma(t))_{t\in\real} \subset \overline{\comp^{-}}$ with the following properties: 
\begin{enumerate}[\rm(A)]
\item\label{d1} $\displaystyle \lim_{t \to \infty}|\gamma(t)|=\lim_{t \to -\infty}|\gamma(t)|=\infty$; 
\item\label{d3} $F_\mu(\gamma) \subset \overline{\comp^-}$; 
\item\label{d4} $F_\mu$ extends to an analytic function in $D(\gamma)$ and to a continuous function on $\overline{D(\gamma)}$, where $D(\gamma)$ denotes the simply connected open set containing $\comp^+$ with boundary $\gamma$;
\item\label{d5} $F_\mu(z) = z+o(z)$ uniformly as $z \to \infty,~ z \in D(\gamma).$ 
\end{enumerate}
\end{proposition}
\begin{proof} 
For $R> |\gamma(0)|$, let $t_1:= \sup\{t<0: |\gamma(t)|\geq R\} \in(-\infty,0)$ and $t_2:= \inf \{t>0: |\gamma(t)| \geq R\}\in (0,\infty)$.  
The circle $\{z\in\comp: |z|=R\}$ is divided into two arcs by $\gamma(t_1),\gamma(t_2)$, and let $A$ be the arc which contains $\{z\in\comp^+: |z|=R\}$. 
Consider a simple closed curve $\gamma_R$ consisting of the arcs $(\gamma(t))_{t\in[t_1,t_2]}$ and $A$. From (\ref{d5}), we can take $R>0$ large enough so that $|F_\mu(z)-z| \leq \frac{1}{2}|z|$ for $z\in D(\gamma),~|z|>R$. 
From the assumption (\ref{d3}), $F_\mu$ maps the simple closed curve $\gamma_R$ to a curve surrounding each point of $\{z\in\comp^+: |z| <\frac{1}{2}R \}$ exactly once, and so the univalent map $F^{-1}_\mu$ can be defined in $\{z\in\comp^+: |z| <\frac{1}{2}R \}$ as the left inverse map of $F_\mu|_{D(\gamma_R)}$ which maps numbers with large positive imaginary parts to numbers with large positive imaginary parts. Here $D(\gamma_R)$ is the bounded Jordan domain surrounded by $\gamma_R$. 
Letting $R\to \infty$, we conclude by analytic continuation that $F_\mu^{-1}$ exists in $\comp^+$ as a univalent map. 
\end{proof}

\begin{figure}[htbp]\label{dia1}
\begin{center}
\includegraphics[width=8.5cm]{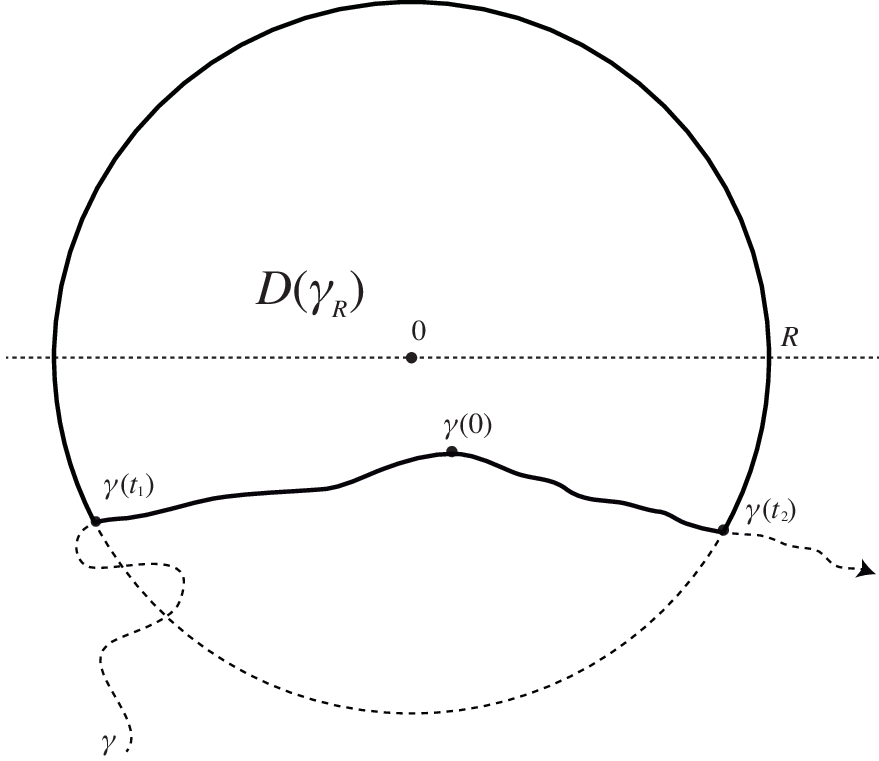}
\end{center}
\end{figure}

\begin{remark}
Note that the map $F_\mu|_{D(\gamma_R)}$ may not be univalent in whole of $D(\gamma_R)$. The fact that each point of $\{z\in\comp^+: |z| <\frac{1}{2}R \}$ has rotation number 1 implies that 
there exists a subset $S_R$ (which is in fact open and simply connected) of  $D(\gamma_R)$ such that $F_\mu$ is univalent in $S_R$ and that $F_\mu(S_R)=\{z\in\comp^+: |z| <\frac{1}{2}R \}$. 

\end{remark}

\section{Proof for Meixner distributions} 
We present some properties of Meixner distributions. 
\begin{enumerate}[\rm(1)]
\item $\rho_t$ is a probability measure for $t>0$ because 
\[
\begin{split}
\int_{\real}|\Gamma(t+ix)|^2\,dx 
&= \int_{\real} \left| \int_{0}^\infty s^{t+ix-1}e^{-s}\,ds \right|^2 \,dx = \int_{\real} \left| \int_{\real} e^{tu -e^{u}}e^{ixu}\,du \right|^2 \,dx\\
&= 2\pi  \int_{\real} e^{2tu -2e^{u}}\,du =2\pi \int_{0}^\infty \left(\frac{s}{2}\right)^{2t} e^{-s}\frac{ds}{s} = \frac{2\pi\Gamma(2t)}{4^t}, 
\end{split}
\]
where Plancherel's theorem was used in the third equality. 

\item $\rho_{1/2}$ coincides with $\mu_1$ thanks to the formula $\Gamma(1-z)\Gamma(z)=\frac{\pi}{\sin(\pi z)}.$ 

\item By the residue theorem,  $G_t:= G_{\rho_t}$ has the series expansion  
$$
G_{t} (z) = \frac{4^t}{\Gamma(2t)}\sum_{n=0}^\infty \frac{(-1)^n \Gamma(n+2t)}{n!}\cdot\frac{1}{z+i(t+n)},  
$$
which is convergent for $0 < t \leq 1/2$. 

\item \label{p3} For any compact set $I \subset \real$, there is $M>0$ such that 
$$
|\Gamma(t+zi)\Gamma(t-zi)| \leq M e^{-\pi |x|}|x|^{2t-1},~~z=x+yi,~|x| \geq 1,~ t,y \in I.  
$$ 
This estimate follows from Stirling's formula. 

\item The density function of $\rho_t$ is symmetric, and moreover strictly decreasing on $[0,\infty)$ as the following calculation shows. We have $\frac{d}{dx} |\Gamma(t+xi)|^2 = -2|\Gamma(t+xi)|^2 \text{Im}\,\psi(t+xi)$ by using the digamma function $\psi(z)=\frac{d}{dz}\log \Gamma(z)$. It is known that $\psi(z)=-\gamma - \sum_{n=0}^\infty\left( \frac{1}{z+n}-\frac{1}{n+1}\right)$, where $\gamma$ is Euler's constant, 
and so $\text{Im}\, \psi(t+xi) =  \sum_{n=0}^\infty\frac{x}{(t+n)^2 +x^2}>0$ for $x>0$. 
\end{enumerate}
We do not use the series expansion of $G_t(z)$; instead the following recursive relation is useful. 
\begin{proposition}\label{prop3}
It holds that 
\begin{equation}\label{eq99}
G_t\left(z-t i\right)=\frac{1}{z}+\frac{it}{z}G_{t+\frac{1}{2}}\left(z+\left(\frac{1}{2}-t\right)i\right),~~\text{\normalfont Im}\,z>t,~~t>0.
\end{equation}
 Iterative use of this relation extends $G_t$ to a meromorphic function in $\comp$ with poles at $-(t+n)i$, $n=0,1,2,\cdots$. 
\end{proposition}
\begin{proof}
Assume $t > \frac{1}{2}$. Because $\Gamma(t+iz)\Gamma(t-iz)$ does not have a pole in $\{z \in \mathbb{C}:  -\frac{1}{2} \leq \text{\normalfont Im}\,z \leq 0\}$ and vanishes rapidly as $\text{Re}\,z \to \infty$ (see the above property (\ref{p3})), 
\[
\begin{split}
G_{t}\left(z-\frac{i}{2}\right)
&=\frac{4^t}{2\pi \Gamma(2t)}\int_{\mathbb{R}}\frac{1}{z-(x+\frac{i}{2})}\Gamma(t+ix)\Gamma(t-ix)\,dx \\
&=  \frac{4^t}{2\pi \Gamma(2t)}\int_{\mathbb{R}}\frac{1}{z-x}\Gamma\left(t+\frac{1}{2} +ix\right)\Gamma\left(t-\frac{1}{2}-ix\right)\,dx,~~\text{Im}\,z >\frac{1}{2}.
\end{split}
\]
By using the basic relation $z\Gamma(z)=\Gamma(z+1)$, we obtain
\[
\begin{split}
G_{t}\left(z-\frac{i}{2}\right)
&=  \frac{4^t}{2\pi \Gamma(2t)}\int_{\mathbb{R}}\frac{\Gamma\left(t+\frac{1}{2} +ix\right)\Gamma\left(t+\frac{1}{2}-ix\right)}{(z-x)(t-\frac{1}{2}-ix)}\,dx \\
&=  \frac{4^t}{2\pi \Gamma(2t)}\int_{\mathbb{R}}\frac{1}{z+(t-\frac{1}{2})i}\left(\frac{1}{t-\frac{1}{2}-ix}-\frac{1}{iz-ix}\right)\left|\Gamma\left(t+\frac{1}{2} +ix\right)\right|^2 \,dx \\
&=  \frac{t i}{z+(t-\frac{1}{2})i}\cdot\frac{4^{t+\frac{1}{2}}}{2\pi \Gamma(2t+1)}\int_{\mathbb{R}}\frac{1}{z-x} \left|\Gamma\left(t+\frac{1}{2} +ix\right)\right|^2 \,dx \\
&~~~~~~~~~+   \frac{1}{(z+(t-\frac{1}{2})i)}\cdot\frac{4^t}{2\pi \Gamma(2t)}\int_{\mathbb{R}}\frac{|\Gamma\left(t+\frac{1}{2} +ix\right)|^2}{t-\frac{1}{2}-ix} \,dx. \\
\end{split}
\]
In the last integral, we can again apply the formula $z\Gamma(z)=\Gamma(z+1)$, and moreover we deform the contour $\mathbb{R}$ to $\mathbb{R}+\frac{i}{2}$: 
\[
\begin{split}
  \frac{4^t}{2\pi \Gamma(2t)}\int_{\mathbb{R}} \frac{|\Gamma\left(t+\frac{1}{2} +ix\right)|^2}{t-\frac{1}{2}-ix} \,dx 
&=   \frac{4^t}{2\pi \Gamma(2t)}\int_{\mathbb{R}}\Gamma\left(t+\frac{1}{2} +ix\right)\Gamma\left(t-\frac{1}{2}-ix\right) \,dx\\
&= \frac{4^t}{2\pi \Gamma(2t)}\int_{\mathbb{R}}\Gamma\left(t+ix\right)\Gamma\left(t-ix\right) \,dx\\
 &= 1. 
\end{split}
\]
The above calculations amount to
$
G_t\left(z-\frac{i}{2}\right)=\frac{1}{z+(t-\frac{1}{2})i}+\frac{it}{z+(t-\frac{1}{2})i}G_{t+\frac{1}{2}}(z),   
$
which holds for any $t>0$ since $G_t(z)$ depends on $t>0$ real analytically. The desired relation (\ref{eq99}) follows from the simple replacement of $z$ by $z+(\frac{1}{2}-t)i$. 
The right hand side of (\ref{eq99}) is meromorphic in $\{z\in\comp: \text{Im}\,z>t-\frac{1}{2}\}$ with pole at $0$, so that $G_t$ extends to a meromorphic function  in $\{z\in\comp: \text{Im}\,z>t-\frac{1}{2}\}$. Next we can write $G_{t+\frac{1}{2}}$ in terms of $G_{t+1}$, and so iteratively $G_t$ can be written in terms of $G_{t+\frac{n}{2}}$ for any $n \in \mathbb{N}$. This procedure extends $G_t$ to a meromorphic function in $\comp$ with poles at $-(t+n)i$, $n=0,1,2,\cdots$. 
\end{proof}

 \begin{lemma}\label{lem3}
If a probability measure $\mu$ has a density $p(x)$ such that $p(x) = p(-x)$, $p'(x) \leq 0$ for a.e.\ $x > 0$ and $\lim_{x\to \infty}  p(x)\log x =0$, then it holds that $\text{\normalfont Re}\, G_\mu(x+yi)>0$ for $x,y>0$. 
 \end{lemma}
 \begin{proof}
 The claim follows from the computation
 \[
 \begin{split}
  \text{\normalfont Re}\, G_\mu(x+yi)
  &=  \int_{\mathbb{R}}\frac{x-u}{(x-u)^2+y^2}p(u)\,du
    = -\frac{1}{2}\int_{\mathbb{R}}\left(\frac{\partial }{\partial u}\log\left((x-u)^2+y^2\right) \right) p(u)\,du \\
  &= \frac{1}{2}\int_{\mathbb{R}}\log\left((x-u)^2+y^2\right)p'(u)\,du \\
  &= \frac{1}{2}\int_{0}^\infty\log\left(\frac{(x+u)^2+y^2}{(x-u)^2+y^2}\right)(-p'(u))\,du  >0,~~x,y>0.
  \end{split}
  \]
  The property $p'(-u)=-p'(u)$ was used at the final equality.
 \end{proof}
 
 \begin{theorem}
 The Meixner distribution $\rho_t$ is in $\mathcal{UI}$ for $0 < t \leq \frac{1}{2}$.   
 \end{theorem}
 \begin{proof}
 We may assume that $0<t<\frac{1}{2}$ since the set $\mathcal{UI}$ is closed with respect to the weak convergence. 
We will check conditions (\ref{d1})--(\ref{d5}) for $F_t(z):=\frac{1}{G_t(z)}$ and $\gamma_t:=\{x-t i: x \in \mathbb{R}\}$. (\ref{d1}) is trivial. 
To prove (\ref{d3}), we use Proposition \ref{prop3}: 
$$
\text{Im}\, G_t\left(x-t i\right)= \frac{t}{x}\text{Re}\, G_{t+\frac{1}{2}}\left(x+\left(\frac{1}{2}-t\right)i\right). 
$$
Since $\frac{d}{dx} |\Gamma(t+\frac{1}{2}+xi)|^2 <0$ for $x>0$, we can apply Lemma \ref{lem3} to the measure $\rho_{t+\frac{1}{2}}$, to assert that $\text{Re}\, G_{t+\frac{1}{2}}\left(x+\left(\frac{1}{2}-t\right)i\right)> 0$ for $x>0$. Hence $\text{Im}\, G_t\left(x-t i\right)>0$ for $x>0$ and also for $x < 0$ by symmetry.   
Hence condition (\ref{d3}) holds since $-ti$ is a pole of $G_t$. 

From Proposition \ref{prop3}, $G_t$ is a meromorphic function and so is $F_t$. If $G_t$ had a zero in $\overline{D(\gamma_{t})}$, there would be a point $z_0 \in \comp^+ \cup \mathbb{R} \setminus\{0\}$ such that $G_{t}(z_0-ti)=0$. This implies $1+ti G_{t+\frac{1}{2}}(z_0+(\frac{1}{2}-t)i)=0$ and so 
$G_{t+\frac{1}{2}}(z_0+(\frac{1}{2}-t)i) = \frac{i}{t} \in \comp^+$. This is a contradiction because $G_{t+\frac{1}{2}}$ maps $\comp^+$ into $\comp^-$. 
Thus  condition (\ref{d4}) is proved.

Condition (\ref{d5}) can be checked as follows. Let $p_t(x)$ be the density function of $\rho_t$. In the integral $\int_{\real}\frac{1}{z-x}\,\rho_t(dx)$, one is allowed to replace the contour $\real$ by $C_t:= \{x-  \frac{3t}{2}i  :  -\infty <x < -\frac{3t}{2}\} \cup \{-\frac{3t}{2} i +\frac{3t}{2} e^{i\theta} :  0 \leq \theta \leq \pi \} \cup \{x -\frac{3t}{2} i  :  \frac{3t}{2} < x < \infty\}$: 
$$
\int_{\real}\frac{1}{z-x}\,\rho_t(dx) = \int_{C_t}\frac{1}{z-w} p_t(w)\,dw. 
$$
Clearly 
$
1 = \int_{\real}\,p_t(x)\,dx = \int_{C_t}p_t(w)\,dw, 
$
so we have $1 -zG_{t}(z) = \int_{C_t} \frac{1}{w-z}\,wp_t(w)\,dw. $ If $z$ tends to $\infty$ satisfying $z \in D(\gamma_t)$, then $1 - z G_{t}(z)$ tends to $0$ by Lebesgue convergence theorem. 
This implies $\left|\frac{F_{t}(z)-z}{z}\right| \to 0$, the conclusion. 
\end{proof}

\begin{remark}
The proof uses the inequality that $\text{Re}\,G_{t+\frac{1}{2}}(x+yi) >0$ for $x,y>0$. If this inequality holds even for negative $y$, then we can prove the free infinite divisibility of $\rho_t$ for $t>\frac{1}{2}$ too. 
\end{remark}

\begin{remark}\label{rem2}
 The free cumulant sequence $(r_n(\mu))_{n =1}^\infty$ of a probability measure $\mu$ with finite moments of all orders can be defined as the coefficients of series expansion of $F_\mu ^{-1}(z)-z$: 
$$
F_\mu^{-1}(z) -z = \sum_{n=1}^\infty \frac{r_n(\mu)}{z^{n-1}}, 
$$
see \cite[Remark 16.18]{NS06}.  
 The free infinite divisibility of $\rho_t$ $(0 < t \leq \frac{1}{2})$ implies that the corresponding free cumulant sequence is conditionally nonnegative definite, i.e.\ the $N\times N$ matrix $(r_{m+n}(\rho_t))_{m,n =1}^N$ is nonnegative definite for any $N\geq 1$; see Theorem 13.16 of \cite{NS06}.\footnote{If a measure $\mu$ has a compact support, the free infinite divisibility is equivalent to the conditional nonnegative definiteness of free cumulants. This equivalence can be extended to a measure with finite moments of all orders when the moment problem is determinate.} If $t=\frac{1}{2}$, the free cumulants up to the $10$th order are given by 
 $$
 (r_2(\mu_2), r_4(\mu_1), r_6(\mu_1), \cdots)= (1,3,38,947,37394,\cdots), ~~r_{2n+1}(\mu_1)=0, ~n\geq 0. 
 $$
 This sequence can be found in \cite{OEIS}. 
\end{remark}

\section{Proof for the logistic distribution}
The free infinite divisibility of the logistic distribution $\mu_2$ is proved with direct computation of the Cauchy transform. 
From residue theorem, it turns out that 
\begin{equation}\label{cosh2}
\begin{split}
G_{\mu_2}(z)
&=\sum_{n=1}^\infty \frac{i}{(z+(n-\frac{1}{2})i)^2} \\
 &=  \sum_{n=1}^\infty \frac{2 x(y+n-\frac{1}{2})}{[x^2+(y+n-\frac{1}{2})^2]^2} +i\sum_{n=1}^\infty \frac{x^2-(y+n-\frac{1}{2})^2}{[x^2+(y+n-\frac{1}{2})^2]^2}, ~~~z=x+yi \in \comp^+. 
\end{split}
\end{equation}
Now we take $\gamma_{1/2}:=\{x-\frac{i}{2}  :  x \in \mathbb{R}\}$. 
The imaginary part of $G_{\mu_2}$ on $\gamma_{1/2}$ can be written as
\[
g(x):= \text{Im}\, G_{\mu_2}\left(x-\frac{i}{2}\right)= \sum_{n=0}^\infty\frac{x^2-n^2}{(x^2+n^2)^2}. 
\]
Fortunately, $g$ can be written by elementary functions.  
\begin{lemma}\label{lem5} The function $g$ is given by 
$
\displaystyle g(x)= \frac{1}{2}\left(\frac{1}{x^2}+\left(\frac{\pi}{\sinh(\pi x)}\right)^2\right). 
$
\end{lemma}
\begin{proof}
It is known that $\frac{1}{\sinh (\pi x)} = \frac{1}{\pi x}-\frac{\pi}{6}x+O(x^3)$ as $x \to 0$, and so $\left(\frac{\pi}{\sinh(\pi x)}\right)^2 = \frac{1}{x^2}+O(1)$, $x \to 0$. 
The poles of $\left(\frac{\pi}{\sinh(\pi x)}\right)^2$ are at $x= n i$ ($n \in \mathbb{Z}$) and the function
$\left(\frac{\pi}{\sinh(\pi x)}\right)^2 - \sum_{n =-\infty}^\infty \frac{1}{(x-n i)^2}$ does not have a singular point. This function is bounded by a constant on $\comp$ and so 
equal to a constant, which is actually zero as is known from the limit $x \to \infty$. Hence 
\[
\begin{split}
\left(\frac{\pi}{\sinh(\pi x)}\right)^2
&= \sum_{n =-\infty}^\infty \frac{1}{(x-n i)^2}
= \frac{1}{x^2} +  \sum_{n=1}^\infty\left(\frac{1}{(x-n i)^2}+ \frac{1}{(x+n i)^2}\right) \\
&= \frac{1}{x^2} +2 \sum_{n=1}^\infty \frac{x^2 -n^2}{(x^2+n^2)^2},  
\end{split}
\]
leading to the conclusion. 
\end{proof}
We easily find that $g(x)>0$ for $x \neq 0$ thanks to Lemma \ref{lem5}, and the function $F_{\mu_2}$ vanishes at $-\frac{i}{2}$ since it is a pole of $G_{\mu_2}$. 
Hence condition (\ref{d3}) is satisfied. 

The following properties can be proved from (\ref{cosh2}): 
\begin{enumerate}[\rm(i)] 
\item $\text{Re}\, G_{\mu_2}(x+yi) > 0$ for $x>0$ and $y \geq -\frac{1}{2}$;
\item $\text{Im}\, G_{\mu_2}(yi) < 0$ for $y > -\frac{1}{2}$. 
\end{enumerate}
So $G_{\mu_2}$ does not have a zero in $\overline{D(\gamma_{1/2})}$ and so $F_{\mu_2}$ is analytic in $D(\gamma_{1/2})$, continuous on $\overline{D(\gamma_{1/2})}$. Consequently 
$\gamma_{1/2} =\{x-\frac{i}{2}  :  x \in \mathbb{R}\}$ satisfies condition (\ref{d4}). 

Condition (\ref{d5}) is proved similarly to the case of $\rho_t$. 

\textbf{Open problems.} 
The authors have not been able to solve the following questions. 
\begin{enumerate}[\rm(a)]
\item Free infinite divisibility for Meixner distributions $\rho_t$  in the case $t > \frac{1}{2}$ and for non symmetric Meixner distributions.  
\item Free infinite divisibility for the measure with density  $\frac{2\pi}{2^r B(\frac{r}{2}, \frac{r}{2})}(\frac{1}{\cosh \pi x})^{r}$ for $r >0$, $r \neq 1,2$. 
\item Characterization of the class $\mathcal{UI}$ in terms of free L\'evy measures. 
\item Combinatorial meaning of the free cumulant sequence of $\rho_t$, in particular of $\rho_{1/2}.$  
\end{enumerate}

\section*{Acknowledgements}
The authors could improve the previous version of the manuscript thanks to referee's kind, helpful suggestions and comments.   
The work was partially supported by the MAESTRO grant DEC-2011/02/A/  ST1/00119 and OPUS grant DEC-2012/05/B/ST1/00626 of National Centre of Science (M.\ Bozejko). 
This work was supported by Global COE Program at Kyoto University and also by Marie Curie International Incoming Fellowship in Universit\'e de Franche-Comt\'e (T.\ Hasebe).


\begin{thebibliography}{99}
\bibitem[AS70]{AS70} M.\ Abramowitz and I.A.\ Stegun, \emph{Handbook of Mathematical Functions with Formulas, Graphs, and Mathematical Tables}, National
Bureau of Standards, Washington, 1970. 
\bibitem[ABBL10]{ABBL10} M.\ Anshelevich, S.T.\ Belinschi, M.\ Bo\.zejko and F.\ Lehner, Free infinite divisibility for Q-Gaussians, Math.\ Res.\ Lett.\ \textbf{17} (2010), 905--916.
\bibitem[AB]{AB} O.\ Arizmendi and S.T.\ Belinschi, Free infinite divisibility for ultrasphericals, Infin.\ Dimens.\ Anal.\ Quantum Probab.\ Relat.\ Top.\ \textbf{16} (2013), 1350001 (11 pages). 
\bibitem[AHa]{AHa} O.\ Arizmendi and T.\ Hasebe. On a class of explicit Cauchy-Stieltjes transforms related to monotone stable and free Poisson laws, Bernoulli, to appear.  
\bibitem[AHb]{AHb} O.\ Arizmendi and T.\ Hasebe, Classical and free infinite divisibility for Boolean stable laws, Proc.\ Amer.\ Math.\ Soc., to appear. arXiv:1205.1575 
\bibitem[AHS]{AHS} O.\ Arizmendi, T.\ Hasebe and N.\ Sakuma, On the law of free subordinators, ALEA, Lat.\ Amer.\ J.\ Probab.\ Math.\ Stat. \textbf{10}, No.\ 2 (2013), 271--291. 
\bibitem[BBLS11]{BBLS11} S.T.~Belinschi, M.\ Bo\.{z}ejko, F.\ Lehner and  R.\ Speicher, The normal distribution is $\boxplus$-infinitely divisible, Adv.\  Math.\ \textbf{226}, No.\ 4 (2011), 3677--3698. 
\bibitem[BV93]{BV93} H.~Bercovici and D.~Voiculescu, Free convolution of measures with unbounded support, Indiana Univ.\ Math.\ J.\ \textbf{42}, No.\ 3 (1993), 733--773.
\bibitem[B92]{B92} L.\ Bondesson, Generalized gamma convolutions and related classes of distributions and densities, Lecture Notes in Stat.\ \textbf{76}, Springer, New York, 1992. 
\bibitem[H]{H} T.\ Hasebe, Free infinite divisibility of measures with rational function densities, preprint. 
\bibitem[KLS10]{KLS10} R.\ Koekoek, P.A.\ Lesky and R.F.\ Swarttouw, \emph{Hypergeometric orthogonal polynomials and their q-analogues}, Springer-Verlag, Berlin, 2010. 
\bibitem[LM95]{LM95} H.\ van Leeuwen and H.\ Maassen,  A q-deformation of the Gauss distribution, J.\ Math.\ Phys.\ \textbf{36} (1995), No.\ 9, 4743--4756.
\bibitem[L51]{L54} P.\ L\'evy, Wiener's random functions, and other Laplacian random functions, Proc.\ 2nd Berkeley Symp.\ on Math.\ Statist.\ and Prob.\ (Univ.\ of California\ Press, 1951), 171--187.
\bibitem[NS06]{NS06} A.\ Nica and R.\ Speicher, {\it Lectures on the Combinatorics of Free Probability}, London Math.\ Soc., Lecture Notes Series \textbf{335}, Cambridge University Press, 2006.
\bibitem[OEIS]{OEIS} OEIS Foundation Inc.\ (2011), The On-Line Encyclopedia of Integer Sequences, http://oeis.org/A158119. 
\bibitem[ST98]{ST98} W.\ Schoutens and J.L.\ Teugels, L\'evy processes, polynomials and Martingales, Commun.\ Statist.-Stoch.\ Mod.\ \textbf{14} (1998) 335--349.
\bibitem[V86]{V86} D.\ Voiculescu, Addition of certain non-commutative random variables, J.\ Funct.\ Anal.\ \textbf{66} (1986), 323--346.
\end{thebibliography}
\end{document}